\documentclass[11pt]{article}
\usepackage{times}
\usepackage{amsmath}
\usepackage{amsthm}
\usepackage{latexsym}
\usepackage{amssymb}
\usepackage{enumerate}
\usepackage{graphicx}
\usepackage{psfrag}
\usepackage{epic}

\newtheorem{theorem}{Theorem}[section]

\newtheorem{example}[theorem]{Example}
\newtheorem{conjecture}[theorem]{Conjecture}
\newtheorem{problem}[theorem]{Problem}
\newtheorem{lemma}[theorem]{Lemma}
\newtheorem{corollary}[theorem]{Corollary}

\begin{document}

\title{  \ \\*[-0.25in] Sign patterns with minimum rank 3 and point-line configurations}

\author{ Guangming Jing$^b$, Wei Gao$^b$, Yubin Gao$^a$, Fei Gong$^b$, \\ Zhongshan Li$^{\rm a, b}$\thanks{Corresponding author. Email: zli@gsu.edu}, 
 Yanling Shao$^a$, Lihua Zhang$^b$  \\
$^a$ {Dept.  of Math, North University of China, Taiyuan, Shanxi, China  }\\
 $^b$ {Dept of Math \& Stat, Georgia State University, Atlanta, GA 30302, USA}  }


\maketitle

\begin{abstract} A \emph{sign pattern (matrix)} is a matrix whose entries are from the set  $\{+, -, 0\}$. The \emph{minimum rank} (respectively, \emph{rational minimum rank}) of a sign pattern matrix $\cal A$ is the minimum of the ranks of the real (respectively, rational) matrices whose entries have signs equal to the corresponding entries of $\cal A$. A sign pattern $\cal A$ is said to be \emph{condensed} if $\cal A$ has no zero row or column and no two rows or columns are identical or negatives of each other. In this paper, a new direct connection between condensed $m \times n $ sign patterns with minimum rank $r$ and $m$ point--$n$ hyperplane configurations in ${\mathbb R}^{r-1}$ is established. In particular, condensed sign patterns with minimum rank 3 are closed related to point--line configurations on the plane. It is proved  that for any sign pattern $\cal A$ with minimum rank $r\geq 3$, if the number of zero entries on each column of $\cal A$ is at most $r-1$, then the rational minimum rank of $\cal A$ is also $r$. Furthermore, we construct the smallest known sign pattern whose minimum rank is 3 but whose rational minimum rank is greater than 3.
\end{abstract}


\section{Introduction}

In combinatorial matrix theory, there has been a lot of interests in the study of sign pattern matrices during the last 50 years (see \cite{Bru95} and \cite{Hall13} and the extensive references therein). A matrix whose entries are from the set $\{+, -, 0\}$ is called a {\it sign pattern (matrix)}. For a real matrix $B$, sgn($B$) is the sign pattern matrix obtained by replacing each positive (respectively, negative, zero) entry of B by + (respectively, $-, 0$). The {\it sign pattern class} (also known as the { \it qualitative class}) of a sign pattern matrix $\cal A$, denoted $Q(\cal A)$, is defined as
$$ Q({\cal A})= \{ A : A \text{ is a real matrix and }\mbox{sgn}(A) = \cal A \}.$$
A square sign pattern ${\cal A}$ is said to be  {\it sign nonsingular} if every matrix $B \in Q({\cal A})$ is nonsingular.

For a sign pattern matrix $\cal A$, the {\it minimum rank} of $\cal A$, denoted mr$(\cal A)$, is the minimum of the ranks of the real matrices in $Q(\cal A)$. Similarly, the {\it maximum rank} of $\cal A$, denoted MR$(\cal A)$, is the maximum of the ranks of the real matrices in $Q(\cal A)$. The {\it rational minimum rank} of a sign pattern $\cal A$, denoted mr$_{\mathbb Q}(\cal A)$, is defined to be the minimum of the ranks of the rational matrices in $Q(\cal A)$. The minimum ranks of sign pattern matrices have been the focus of considerable research in recent years, see for example \cite{Arav05}-\cite{Berman08}, \cite{Bru10}-\cite{Gra96}, and \cite{Hall04}-\cite{Raz10}. 

A square $n \times n$ sign pattern is called a {\it permutation sign pattern} if each row and column contains exactly one $+$ entry and  $n-1$ zero entries.  Two $m\times n$ sign pattern matrices ${\cal A}_1$ and ${\cal A}_2$ are said to be {\it permutationally equivalent} if there exist permutation sign patterns $P_1$ and $P_2$ such that ${\cal A}_2 =P_1 {\cal A}_1 P_2$. A {\it signature sign pattern} is a square diagonal sign pattern all of whose diagonal entries are nonzero. Two $m\times n$ sign patterns ${\cal A}_1$ and ${\cal A}_2$ are said to be {\it diagonally equivalent} or {\it signature equivalent}  if there exist signature sign patterns $D_1$ and $D_2$ such that ${\cal A}_2 =D_1 {\cal A}_1 D_2$. It is easy to observe that if two sign patterns are permutationally equivalent or diagonally equivalent, they have the same minimum rank and the same rational minimum rank.

Consider a nonzero sign pattern $\cal A$. If $\cal A$ contains a zero row or column, then deletion of the zero row or zero column preserves the  minimum rank. Similarly, if two nonzero rows (or two nonzero columns) of $\cal A$ are either identical or are negatives of each other, then deleting such a row (or column) also preserves the minimum rank.

Following  \cite{Li13}, we say that a sign pattern is a \emph{condensed sign pattern} if does not contain a zero row or zero column and no two rows or two columns are identical or are negatives of each other. Clearly, given any nonzero sign pattern $\cal A$, we can delete zero, duplicate or opposite rows and columns of $\cal A$ to get a condensed sign pattern matrix ${\cal A}_c$ (called the condensed sign pattern of $\cal A$). For consistency, we agree that when two rows are identical or opposite, we delete the lower one and when two columns are identical or opposite, we delete the one on the right. For a zero sign pattern $\cal A$, the condensed sign pattern ${\cal A}_c =\emptyset $, the empty matrix, which has minimum rank 0.

Obviously, $\text{mr} (\cal A) \leq  \text{mr}_{\mathbb Q} (\cal A)$ for every sign pattern $\cal A$.  In \cite{Arav05}, \cite{Arav09} and \cite{Arav13}, several classes of sign patterns $\cal A$ such that  $\text{mr} (\cal A) = \text{mr}_{\mathbb Q} (\cal A)$ are identified, such as when $\cal A$ is entrywise nonzero, or the minimum rank of $\cal A$ is at most 2, or $\mbox{MR}({\cal A})-\mbox{mr}({\cal A}) \leq 1$, or the minimum rank of $\cal A$ is at least $\min\{m, n\}-2$, where $\cal A$ is $m \times n$.

However, it has been shown in \cite{KP08} and \cite{Berman08} that there exist sign patterns $\cal A$ for which $\text{mr} ({\cal A}) \leq  \text{mr}_{\mathbb Q} (\cal A)$. In particular, \cite{KP08} showed the  existence of a $12\times 12$ sign pattern with  $\text{mr} ({\cal A}) =3$ but $  \text{mr}_{\mathbb Q} ({\cal A}) >3.$ 

In this paper, we establish a new direct connection between condensed $m \times n $ sign patterns with minimum rank $r  \ (r\geq 2)$ and $m$ point--$n$ hyperplane configurations in ${\mathbb R}^{r-1}$. Then we use the matrix factorization that guarantees this connection to prove that if the number of zero entries on each column of a sign pattern $\cal A$ with minimum rank $r$ is at most $r-1$ , then $\text{mr} (\cal A) = \text{mr}_{\mathbb Q} (\cal A)$. Furthermore, we construct the smallest known sign pattern whose minimum rank is 3 but whose rational minimum rank is greater than 3.
We note that as shown in the next section, rational realizability of the minimum rank for sign patterns with minimum rank 3 is closed related to the central problem of rational realizability of certain point--line configurations on the plane \cite{Grunbaum09}.


\section{point--hyperplane configuration}

We now establish a direct connection between $m \times n $ condensed sign patterns with minimum rank $r$ ($r\geq 2$) and $m$ point--$n$ hyperplane  configurations in ${\mathbb R}^{r-1}$.

To create this connection, we need the following lemma. 

\begin{lemma}\label{lem:B=UV}
Let $\cal A$ be an $m\times n$ condensed sign pattern with mr$({\cal A})= r \ge 2$. Then there are suitable signature sign patterns $D_1$ and $D_2$, such that there is a real matrix $B \in Q(D_1{\cal {A}}D_2)$ with $\mbox{rank}(B)=r$ such that $B=UV$, where $U$ is $m\times r$, V is $r\times n$, and
$$ U= \begin{bmatrix} 1 & u_{12} & \cdots & u_{1r}  \\
1 & u_{22} & \cdots & u_{2r}  \\
 \vdots & \vdots & \vdots & \vdots \\
 1 & u_{m2} & \cdots & u_{mr}  \end{bmatrix},
\quad \mbox{and } \ V=\begin{bmatrix}
 v_{11} & v_{12} & \cdots & v_{1n} \\
 v_{21} & v_{22} & \cdots & v_{2n} \\
 \vdots & \vdots & \vdots & \vdots \\
  v_{r-1,1} & v_{r-1,2} & \cdots & v_{r-1,n} \\
1 &  1  & \cdots & 1 			\end{bmatrix}.  $$	

\end{lemma}

\begin{proof}

Let $B_0 \in Q(\cal A)$ with $\mbox{rank}(B_0) = r$. Then there exist an $m \times r$ matrix $U_0$ and an $r \times n$ matrix $V_0$ such that $B_0 = U_0V_0$. Because $\cal A$ is a condensed sign pattern, no two rows of $U_0$ are linearly dependent and no two columns of $V_0$ are linearly dependent.

Since there are only finitely many rows of $U_0$ and finitely many columns of $V_0$, there is a suitable rotation matrix $R(\theta_2 ;1,2)$ of order $r$ such that

(i) if the first two components of the $i$th row of $U_0$ are not both zero, then the first component of the $i$th row of $U_0R(\theta_2 ;1,2)^T$ is nonzero, and

(ii) if the first two components of the $j$th column of $V_0$ are not both zero, then the first component of the $j$th column of $R(\theta_2 ;1,2)V_0$ is nonzero, for all $i$ and $j$.

Continuing in this fashion, we can find suitable $\theta_3, \cdots, \theta_r $ such that for each $k$, $2\leq k \leq r$,

(i) if the first $k$ components of the $i$th row of $U_0$ are not all zero, then the first component of the $i$th row of $U_0 R(\theta_2 ; 1, 2)^T R(\theta_3 ; 1, 3)^T \cdots R(\theta_k ; 1, k)^T $ is nonzero, and

(ii) if the first $k$ components of the $j$th column of $V_0$ are not all zero, then the first component of the $j$th column of $ R(\theta_k ; 1, k) \cdots R(\theta_3 ; 1, 3)R(\theta_2 ; 1, 2) V_0$ is nonzero, for all $i$ and $j$.

Furthermore, $\theta_r$ may be chosen to ensure that the last component of each column of  $R(\theta_r ; 1, r)$ $\cdots R(\theta_3 ; 1, 3)R(\theta_2 ; 1, 2) V_0$ is nonzero. Since $U_0$ has no zero row, the first component of each row of $U_0 R(\theta_2 ; 1, 2)^T R(\theta_3 ; 1, 3)^T $  $\cdots R(\theta_r ; 1, r)^T $ is nonzero.

Let $Q = R(\theta_r ; 1, r) \cdots R(\theta_3 ; 1, 3)R(\theta_2 ; 1, 2)$. By replacing $B_0 = U_0V_0$ with $B = (D_1 U_0 Q^T)$  $(Q V_0 D_2)$, $U_0$ with $U = D_1 U_0 Q^T$, $V_0$ with $V = Q V_0 D_2$, and $\cal A$ with a diagonally equivalent sign pattern sgn$(D_1) {\cal {A}} \text{ sgn}(D_2)$ for some suitable nonsingular diagonal matrices $D_1$ and $D_2$ if necessary, we may assume that the first entry of each row of $U$ is 1, and the last entry of each column of $V$ is 1. Thus we arrive at the desired factorization $B = UV$.
\end{proof}

With the previous lemma, we can construct an $m$ point--$n$ line configuration in the plane for every $m \times n $ condensed sign pattern with minimum rank 3. Let $\cal A$ be an $m\times n$ condensed sign pattern with $\mbox{mr}({\cal A})=3$. Then we may assume that there is a matrix $B \in Q(\cal A)$ with $\mbox{rank}(B) = 3$ such that $B = UV$, where
$$ U= \begin{bmatrix} u_1  \\
u_2  \\
 \vdots \\
 u_m \end{bmatrix}
 =\begin{bmatrix} 1 & u_{12} & u_{13}  \\
1 & u_{22}  & u_{23}  \\
 \vdots & \vdots  & \vdots \\
 1 & u_{m2}  & u_{m3}  \end{bmatrix}
, $$
$$ V= \begin{bmatrix}
 v_1 & v_2 & \cdots & v_n 	\end{bmatrix}
=\begin{bmatrix}
 v_{11} & v_{12} & \cdots & v_{1n} \\
 v_{21} & v_{22} & \cdots & v_{2n} \\
1 &  1  & \cdots & 1 			\end{bmatrix}.  $$	

Identify the $i$th row of $U$, $(1, u_{i2}, u_{i3})$, with the point $p_i= ( u_{i2}, u_{i3})$ in $\mathbb R^2$, for $1 \leq i \leq m$. Identify the $j$th column of $V$, $(v_{1j}, v_{2j}, 1)^T$ with the straight line $l_j$ in $\mathbb R^2$ given by the equation  $v_{1j}+x v_{2j}+y = 0$, for $1 \leq j \leq n$. After taking such identifications, these resulting  $m$ points and $n$ lines form an $m$ point--$n$ line configuration (unlike in \cite{Grunbaum09}, here and throughout the paper, we use the word configuration in the sense of a solid figure, with no further starting assumptions) in the plane that satisfies

(i) \phantom{i} $b_{ij} > 0$ if and only if the point $p_i$ is above the line $l_j$;

(ii)\phantom{i} $b_{ij} = 0$ if and only if the point $p_i$ is on the line $l_j$;

(iii)  $b_{ij} < 0$ if and only if the point $p_i$ is below the line $l_j$.

Conversely, every $m$ point--$n$ line configuration in the Euclidean plane $\mathbb R^2$ gives rise to an $m \times n $ sign pattern with minimum rank at most 3.
Let $C$ be a configuration in the plane with $m$ labeled points $p_1, p_2, \cdots, p_m$ and $n$ labeled lines $l_1, l_2, \cdots, l_n$. By taking a suitable rotation (whose effect on the resulting sign pattern will be explained later) if necessary,  we may assume that there is no vertical line in $C$.

Let ${\cal A}=[{\bold a}_{ij}]$ be an $m \times n $ sign pattern such that

(i) \phantom{i} ${\bold a}_{ij} = +$ if and only if the point $p_i$ is above the line $l_j$;

(ii) \phantom{i}${\bold a}_{ij} = 0$ if and only if the point $p_i$ is on the line $l_j$;

(iii) ${{\bold a}}_{ij} = -$ if and only if the point $p_i$ is below the line $l_j$.

\noindent Then $\cal A$ is an $m \times n$ sign pattern corresponding to $C$ and $\mbox{mr}({\cal A}) \leq 3$. Indeed, by interpreting each point $p_i=(u_{i2}, u_{i3})$ as a row $[1,u_{i2}, u_{i3}]$ of an $m\times 3$ matrix $U$, and interpreting each line $l_j$ with equation $v_{1j} + v_{2j} x + y =0$ as a column $[v_{1j}, v_{2j}, 1]^T$ of a $3\times n$ matrix $V$, we obtain a real matrix $A=UV$ with rank$(A) \leq 3$ and sgn$(A)=\cal A$. 

For example, let $C$ be the following point--line configuration.

\hspace{7.3cm}
\setlength{\unitlength}{0.8mm}
\begin{picture}(100,60)

\put(-5,40){\makebox(0,0){$\bullet$}}
\put(10,30){\makebox(0,0){$\bullet$}}
\put(-10,10){\makebox(0,0){$\bullet$}}

\put(-40,10){\line(5,2){80}}
\put(-25,-5){\line(1,1){55}}
\put(-30,50){\line(1,-1){55}}

\put(-1,44){\makebox(0,0){$p_1$}}
\put(14,26){\makebox(0,0){$p_2$}}
\put(-6,7){\makebox(0,0){$p_3$}}

\put(-34,50){\makebox(0,0){$l_1$}}
\put(-29,-5){\makebox(0,0){$l_2$}}
\put(-42,7){\makebox(0,0){$l_3$}}

\put(0,-10){\makebox(0,0){Figure $2.1$}}

\end{picture}

\vspace{12mm}

\noindent Then the corresponding sign pattern matrix $\cal A$ is 
$${\cal A}= \begin{bmatrix} + & + & +  \\
+ & 0 & 0 \\
- & 0 & -
\end{bmatrix}.$$

The point $p_1$ is above all 3 lines. So on the first row of $\cal A$ there are 3 positive signs. Similarly, we can get the second row and the third row.

It is useful to think of each line in a point--line configuration on the plane to be directed (namely, oriented), so that the $(i,j)$-entry of the resulting sign pattern is $+$ (respectively, $-, 0$) if and only if  the point $p_i$ is on the left (respectively,  right, inside) of the line $l_j$. Then it is clear that reversing the direction of a line in the configuration corresponds to negating a column of the resulting sign pattern. Further, for convenience, we usually assume that each non-vertical line is pointing to the right and each vertical line is pointing upward. 

However, in view of the incidence and orientation preserving dual transform $D_0$ (see \cite{Matou02}) that sends every nonzero point $a\in \mathbb R^2$ to the line  $h_a=D_0(a)= \{ x \in \mathbb R^2 : \langle a, x \rangle = 1\}$ and also sends the line  $h_a$ (which does not pass though the origin) to the point $a$, it is more natural to orient every line not passing through the origin in the clockwise direction, relative to the origin. Of course, the transform $D_0$ maps an $m$ point--$n$ line configuration on the plane to an $n$ point--$m$ line configuration, with their corresponding sign patterns being transposes of each other (assuming all the lines are oriented clockwise relative to the origin). This transform is also defined in $\mathbb R^d$, and the negative side of the hyperplane $h_a=D_0(a)= \{ x \in \mathbb R^d : \langle a, x \rangle = 1\}$ is the halfspace containing the origin. 

Assuming that all nonvertical lines are oriented to point to the right, and the vertical lines are oriented upward, then 
it is clear that any translation will preserve the resulting sign pattern matrix of a point--line configuration. Thus we can translate the configuration to a position so that none of the points in the configuration is the origin and none of the lines in the configuration passes through the origin. It is then apparent that for such a point--line configuration, any rotation of the point--line configuration through the origin preserves the resulting sign pattern. It can be seen that more generally,  
 if two point--line configurations can be obtained from each other through rotation and translation, then their resulting sign patterns are equivalent (through  permutation and signature equivalence).  We say that two point--line configurations are {\it equivalent} if their resulting sign patterns are equivalent.

Furthermore, in order that a point--line configuration on the plane produces a condensed sign pattern, further conditions must be met. It is easy to see that an  $m$ point--$n$ line configuration $C$ results in a condensed sign pattern if and only if the following four  conditions are satisfied for the points and lines in $C$. 
\begin{enumerate}
\item No two points in $C$ have identical or opposite relative positions (above, below, or on) relative to all the $n$ lines in $C$.

\item No two lines in $C$ have the same or opposite relative  positions relative to all the $m$ points in $C$.

\item No point in $C$ is on all the $n$ lines in $C$.
\item No line in $C$ passes through all the $m$ points in $C$.
\end{enumerate}
Such a point--line configuration is said to be {\it simple}. Obviously, a simple point--line configuration gives rise to a sign pattern with minimum rank 1 if and only if it has exactly 1 point and 1 line.

In \cite{Li13}, sign patterns with minimum rank 2 are characterized.

\begin{theorem} \label{thm:mr2char1} \emph{\cite{Li13}} A sign pattern  matrix $\cal A$ has minimum rank 2 if and only if  its condensed sign pattern ${\cal A}_c$ satisfies the following conditions:
\begin{enumerate}
  \item[(i) ] ${\cal A}_c$ has at least two rows and two columns,
  \item[(ii)]  each row and each column of ${\cal A}_c$ has at most one zero entry, and
  \item[(iii)] there are signature sign patterns $D_1$ and $D_2$  and permutation sign patterns $P_1$ and $P_2$ such that each row and each column of $P_1D_1 {\cal A}_c D_2P_2$ is nondecreasing.
 \end{enumerate}
\end{theorem}

In the above theorem, we say a row or the column of a sign pattern is {\it nondecreasing} if the $-$ entries appear before the 0 entries  and all 0 entries appear before the $+$ entries.

From the proof of the preceding theorem in \cite{Li13} using a similar factorization as in Lemma  \ref{lem:B=UV}, it is easy to see that a sign pattern with minimum rank 2 corresponds to a point--point configuration on the line $\mathbb R^1$; such a configuration can be regarded as a degenerate point--line configuration $C$ on the plane in which  all the lines in the configuration $C$ are parallel to each other 
(or equivalently, when all the points in the configuration $C$ are collinear (but the line containing all the points may not be in $C$)). Thus we have the following result.

\begin{theorem} \label{thm:mr2char}
A simple point--line configuration $C$ on the plane gives rise to a sign pattern with minimum rank 2 if and only if it satisfies the following three conditions:
\begin{enumerate}
\item[(i) ] in $C$ there are at least 2 points and 2 lines;
\item[(ii)] each point in $C$  is on at most one  line in $C$ and each line in $C$ passes through at most 1 point in $C$; and 
 \item[(iii)]  $C$ is equivalent to a simple point--line configuration $C_1$  all of whose lines are parallel to each other. 
 \end{enumerate}
\end{theorem}

The following is such a simple configuration yielding a sign pattern with minimum rank 2.  

\hspace{7.3cm}
\setlength{\unitlength}{0.8mm}
\begin{picture}(100,60)

\put(0,0){\makebox(0,0){$\bullet$}}
\put(0,10){\makebox(0,0){$\bullet$}}
\put(0,20){\makebox(0,0){$\bullet$}}
\put(0,30){\makebox(0,0){$\bullet$}}
\put(0,40){\makebox(0,0){$\bullet$}}
\put(0,50){\makebox(0,0){$\bullet$}}

\put(-30,40){\line(1,0){60}}
\put(-30,14){\line(5,1){60}}
\put(-30,6){\line(5,-1){60}}

\put(0,-10){\makebox(0,0){Figure $2.2$}}

\end{picture}
\vspace{8mm}


More generally, in a similar fashion, for every $m \times n $ sign pattern with minimum rank $r \geq 2$, we can construct an $m$ point--$n$ hyperplane configuration in ${\mathbb R}^{r-1}$ using the factorization given in Lemma \ref{lem:B=UV}. Conversely, from  an $m$ point--$n$hyperplane configuration in ${\mathbb R}^{r-1}$ (in which no hyperplane is vertical (namely, parallel to the $x_d$-axis)), we can write out an $m \times n $ sign pattern whose minimum rank is at most $r$. Of course, by saying that a point $p$ is above a hyperplane $H$ in $\mathbb R^d$, we mean that the $x_d$ coordinate of $p$ is greater than the $x_d$ coordinate of the vertical (namely, parallel to the $x_d$-axis) projection of $p$ on $H$.

We give an application of this point--hyperplane configuration approach in the proof of the following theorem. We say that a sign pattern $\cal A$ has a \emph{direct point-hyperplane representation} if a minimum rank factorization  for $\cal A$ given  in Lemma \ref{lem:B=UV} can be done with $D_1$ and $D_2$ being identity sign patterns. For example, for the sign patterns 
$$ {\cal A}_1= \bmatrix + & + &+ \\  - & + & + \\ - & 0 & + \endbmatrix  \quad \text{ and } \quad {\cal A}_2 \bmatrix + & + &+ \\  - & + & + \\ + & 0 & - \endbmatrix, $$ 
it can be easily verified that  both ${\cal A}_1$  and ${\cal A}_2$ have minimum rank 2,  ${\cal A}_1$  has a direct point-hyperplane representation, but ${\cal A}_2$ does not.

\begin{theorem} \label{thm:A_1A_2}
Let ${\cal A}_1$ and ${\cal A}_2$ be two sign patterns that have direct point-hyperplane representations. Suppose that  mr$({\cal A}_1)=r_1\geq 2$ and mr$({\cal A}_2)=r_2\geq 2$. 
Then 
$$\text{mr} \left ( \bmatrix {\cal A}_1 & + \\ - & {\cal A}_2 \endbmatrix \right ) = \max \{ r_1, r_2\}. $$ 
\end{theorem}

\begin{proof} Let ${\cal A} = \bmatrix {\cal A}_1 & + \\ - & {\cal A}_2 \endbmatrix$. Since each of ${\cal A}_1$ and ${\cal A}_2$ is a submatrix of $\cal A$, it is obvious that mr$({\cal A}) \geq \max \{ r_1, r_2\}. $  To complete the proof, we need to show the opposite inequality. 

Without loss of generality, assume that $r_1 \leq r_2$ and let $d=r_2-1$. In a minimum rank factorization $A_1=UV$ given in Lemma \ref{lem:B=UV} of some matrix  $A_1 \in Q({\cal A}_1)$, we may insert $r_2-r_1$ zero columns in $U$ after the first column and also insert as many zero rows in $V$ after the first row. The resulting new factorization $A_1 =U_1 V_1$ can give rise to a point--hyperplane configuration $C_1$ in $\mathbb R^{d}$ that corresponds to ${\cal A}_1$.   From the hypothesis, we can also get a point--hyperplane configuration $C_2$ in $\mathbb R^{d}$ that corresponds to ${\cal A}_2$. The hyperplanes in $C_1$ divide $\mathbb R^{d}$ into connected open regions, one of which consists of all the points in $\mathbb R^d$ that are below all the hyperplanes in $C_1$. This unbounded region is called the lowest region of the arrangement of hyperplanes in $C_1$. Similarly, the arrangement of hyperplanes in $C_2$ has a highest (unbounded) region, consisting of all points in $\mathbb R^d$ that are above all the hyperplanes in $C_2$. Since translation of a configuration does not affect the resulting sign pattern, we may assume that $C_1$ is placed ``far above'' $C_2$, in the sense that all the points (of course, not the hyperplanes) of $C_1$ are in  the highest region of $C_2$, and all the points of $C_2$ are in the lowest region of $C_1$. It is then clear that the point--hyperplane configuration $C_1\cup C_2$ yields the sign pattern ${\cal A} = \bmatrix {\cal A}_1 & + \\ - & {\cal A}_2 \endbmatrix$. The fact that this representation is possible ensures that we can get a factorization of a matrix $A \in Q({\cal A})$ of the form $A=U_0 V_0$, where $U_0$ has $d+1$ columns. It follows that    mr$({\cal A}) \leq d+1 = r_2 =\max \{ r_1, r_2\}. $  This completes the proof. 
\end{proof}

Repeated application of the preceding theorem yields the following result. 

\begin{corollary}
Let ${\cal A}_1, \dots,  {\cal A}_k$ be $k$ ($k\ge 2$) sign patterns that have direct point-hyperplane representations. Suppose that  mr$({\cal A}_i)=r_i\geq 2$  for each $i, 1\leq i \le k$. 
Then 
$$\mbox{mr} \left ( \bmatrix 
{\cal A}_1 & +  & \hdots & +	\\
 - & {\cal A}_2 & \ddots & \vdots \\
\vdots & \ddots & \ddots & +  \\
- & \hdots & - & {\cal A}_k 
\endbmatrix \right ) = \max \{ r_1, \dots,  r_k\}. $$ 
\end{corollary}


\section{Sign patterns with few zeros on each column}

Analyzing a minimum rank factorization given in Lemma \ref{lem:B=UV} for a condensed   $m \times n $ sign pattern with minimum rank 3, we can establish the following result. 

\begin{theorem}\label{thm:2z}

Let $\cal A$ be a condensed $m \times n$ sign pattern with $mr ({\cal A}) = 3$. If the number of zero entries on each column of $\cal A$ is at most 2, then $mr_{\mathbb Q} ({\cal A}) = 3$.

\end{theorem}

\begin{proof}
Without loss of generality, we assume that $\cal A$ has a direct point--line representation. By Lemma \ref{lem:B=UV},  we then have a special minimum rank factorization $A=UV$ of a certain matrix $A=[a_{ij}]\in Q({\cal A})$, where 
$$ U= \begin{bmatrix} 1 & u_{12} & u_{13}  \\
1 & u_{22}  & u_{23}  \\
 \vdots & \vdots  & \vdots \\
 1 & u_{m2}  & u_{m3}  \end{bmatrix}
\quad \mbox{and } \ V= \begin{bmatrix}
 v_{11} & v_{12} & \cdots & v_{1n} \\
 v_{21} & v_{22} & \cdots & v_{2n} \\
1 &  1  & \cdots & 1 			\end{bmatrix}.  $$	

This factorization yields     
 an $m$ point--$n$ line configuration $C$ in the plane corresponding to $\cal A$. We now treat the entries $u_{ij}$ and $v_{ij}$ not in  the first column of $U$ or the last row of $V$ as independent variables allowed to take real values around the initial values in $U$ or $V$. Thus $A=UV$ becomes a matrix whose entries are polynomial functions of the variables  $u_{ij}$ and $v_{ij}$. 
 
 Let $P = \{p_1, p_2, \dots, p_m\}$ and $L = \{l_1, l_2, \dots, l_m\}$ be the points and lines in $C$. As the number of zero entries in each column of $\cal A$ is at most 2, each line $l_j$ contains at most two points in $P$. 
Let $S_j =\{ i : 1\leq i \leq m \text{ and }  p_i \text{ is on } l_j \}$. Then $|S_j| \leq 2$.

For a fixed $j$ ($1\leq j \leq n)$, suppose that $S_j \neq \emptyset$. We solve for certain entries of $V$ in terms other entries of $V$ in the same column and entries of $U$ (not in the first column).   We separate  the following two cases.

Case 1. $|S_j| =1$. \\
Write $S_j=\{ i \} $.  Then $a_{ij}=0$ means that  $v_{1j} + u_{i2}v_{2j} + u_{i3} = 0$. Solving the last equation for $v_{1j}$, we get 
      $$  v_{1j} =  - u_{i2}v_{2j} - u_{i3} .  \eqno{(3.1)} $$

Thus we can regard $v_{1j}$  as a rational function (in fact, a polynomial function in this case) of the entries involved on the right side of the above equation.

Case 2.  $|S_j| =2$.   \\
Write $S_j =\{ k, l \}.$ Then $a_{kj}=0$ and $a_{lj}=0$ mean that 
       $$\left\{
       \begin{array}{ll}
           v_{1j} + u_{k2}v_{2j} + u_{k3} = 0 \\
           v_{1j} + u_{l2}v_{2j} + u_{l3} = 0
        \end{array}.
        \right.
        $$

Note that if $u_{k2} =u_{l2}$, the above equations would imply $u_{k3} =u_{l3}$, so that the $k$-th and $l$-th rows of $U$ are the same, which would imply that $A$ has two equal rows, contradicting the fact that $\cal A$ is a condensed sign pattern. Thus $u_{k2} \ne u_{l2}$. Therefore, we can solve the above system of two equations for $v_{1j}$ and $v_{2j}$ to  obtain

       $$\left\{
       \begin{array}{ll}
           v_{1j} = -u_{k3} + u_{k2}\frac{u_{l3} - u_{k3} }{u_{l2} - u_{k2}} \\
           v_{2j} = \frac{u_{l3} - u_{k3} }{u_{k2} - u_{l2}}
        \end{array}
         \right.  
        \eqno{(3.2)}
        $$

        Thus we can regard $v_{1j}$ and $v_{2j}$ as rational functions of the entries involved on the right side of equations in (3.2).

The continuity of the functions given in (3.1) and (3.2) ensure that we can let the independent variables take suitable rational values sufficiently close to their initial values in $U$ or $V$ to yield rational perturbations $\tilde{U}$ and $\tilde{V}$ of $U$ and $V$ respectively, so that the resulting rational matrix $\tilde{U} \tilde{V}$ is in $Q(\cal A)$ and has rank at most 3. Indeed,  the dependence relations given by (3.1) and (3.2) ensure that $(\tilde{U} \tilde{V})_{ij}=0$ whenever $a_{ij}=(UV)_{ij}=0$. Thus $A_1 = \tilde{U} \tilde{V}$ is a rational realization of the minimum rank of $\cal A$.
\end{proof}

The following result is a significant generalization of the preceding theorem.

\begin{theorem}\label{thm:(r-1)zero}

Let $\cal A$ be any sign pattern and let $r=\text{mr}({\cal A})$. Suppose that the number of zero entries in each column of ${\cal A}_{c}$ is at most $r-1$. Then $mr_{\mathbb Q} ({\cal A}) = mr({\cal A})$ .

\end{theorem}

\begin{proof}
Since $\cal A$ and its condensed sign pattern ${\cal A}_{c}$ have the same minimum rank and the same rational minimum rank, without loss of generality, we may assume that ${\cal A}= {\cal A}_c$. 

Assume that ${\cal A}=[\bold a_{ij}]= {\cal A}_{c}$ is $m \times n$. Further, replacing $\cal A$ with a suitable sign pattern diagonally equivalent to $\cal A$ if necessary, we may assume that $A$ has a direct point-hyperplane representation. 
By Lemma \ref{lem:B=UV},  there is a real matrix $A \in Q({\cal A}_{c})$ with $\mbox{rank}(A)=r$ such that $A=UV$, where $U$ is $m\times r$, V is $r\times n$, 

$$ U= \begin{bmatrix} 1 & u_{12} & \cdots & u_{1r}  \\
1 & u_{22} & \cdots & u_{2r}  \\
 \vdots & \vdots & \vdots & \vdots \\
 1 & u_{m2} & \cdots & u_{mr}  \end{bmatrix},
\quad \mbox{and } \ V=\begin{bmatrix}
 v_{11} & v_{12} & \cdots & v_{1n} \\
 v_{21} & v_{22} & \cdots & v_{2n} \\
 \vdots & \vdots & \vdots & \vdots \\
  v_{r-1,1} & v_{r-1,2} & \cdots & v_{r-1,n} \\
1 &  1  & \cdots & 1 			\end{bmatrix}.  $$

We regard $u_{ij}$ ($1 \leq i \leq m$, $2 \leq j \leq r$) and $v_{kl}$ ($1 \leq k \leq r-1$, $1 \leq l \leq n$) as independent variables that are allowed to take on real number values (close to their original values in $U$ and $V$). Then each nonzero entry of $A=UV$ can be regarded as a polynomial (and hence, continuous) function of the variables $u_{ij}$ and $v_{ij}$. Thus there is a real number $\epsilon >0 $ such that for all values $u_{ij}'$ and $v_{kl}'$ such that $|u_{ij}'-u_{ij}|< \epsilon$ and $|v_{kl}'-v_{kl}|< \epsilon$ for all $i, j, k$ , and $l$ (which gives an open hypercube in $\mathbb R^{(m+n)(r-1)}$, denoted $HC_\epsilon$), then  sgn$(U'V')_{ij}$=sgn$(UV)_{ij}$ whenever $\bold{a}_{ij}\ne 0$. 


Note that the determinant of each $t\times t$ ($1 \le t \le r-1$) square submatrix of $U$ involving the first $t$  columns of $U$  is a nonzero polynomial. Let $D$ be the product of all such determinants. Then $D$ is  a nonzero polynomial in the $ m (r-2)$ variables $u_{ij}$ (with $2\le j \le r-1$).  Let $c= m(r-2)$ and let $d$ be the degree of $D$.   It can be shown by induction on the number of variables (indeterminates) in $D$ that if $D$ is equal to zero at all points  in $Q_1 \times  \dots\times  Q_c$, for some subsets 
$Q_j \subset \mathbb Q$ with cardinality $d+1$, then $D=0$, a contradiction. Thus for any subsets $Q_j\subset \mathbb Q$ with $| Q_j| \geq d+1$, there is at least one point in $Q_1 \times  \dots\times  Q_c$ at which $D\ne 0$. 
It follows  that the set $T$ of rational points in $\mathbb Q^c$ at which $D\neq 0$ is dense in $\mathbb R^c$.  

For a fixed $j$,  assume there are $s_j\ge 1$ zero entries in the $j$-th column of $\cal A$, with row indices $i_1, \dots , i_{s_j}$.  
The fact that $a_{i_t j}=0$ for $1\le t \le  s_j$ means that we have the following system of equations. 
$$\left\{
       \begin{array}{ll}
           v_{1j} + u_{i_12}v_{2j} + \cdots + u_{i_1, r-1}v_{r-1, j} + u_{i_1r} = 0 \\
           v_{1j} + u_{i_22}v_{2j} + \cdots + u_{i_2, r-1}v_{r-1, j} + u_{i_2r} = 0 \\
           \ \ \ \ \ \ \ \vdots \\
           v_{1j} + u_{i_{s_j}2}v_{2j} + \cdots + u_{i_{s_j}, r-1}v_{r-1, j} + u_{i_{s_j}r} = 0
        \end{array}
        \right.  
\eqno{(3.3)}
$$

When the entries of the columns 2 through $r-1$ of $U$ are restricted to the set $T \subset \mathbb Q^c$, we see that for fixed $j$, the above system of equations (3.3) has an invertible coefficient matrix, with $v_{1j}, \dots, v_{s_j j}$  regarded as the unknowns. Thus we can solve for the variables $v_{1j}, \dots, v_{s_j j}$, which are then given as rational functions depending on some variables in $U$ and the remaining entries in the $j$-th column of $V$. After this is done for all $j$ for which there are some zero entries in the $j$-th column of $\cal A$, we have then expressed certain entries of $V$ as rational functions (with all coefficients in $\mathbb Q$) of some  entries of $U$ and the remaining entries in $V$ which are independent variables. Since $T$ is dense in $\mathbb R^c$, we can find a rational point $x_0 \in HC_\epsilon$ such that the projection of $x_0$ on $\mathbb R^c$ is in $T$. We then treat the components of 
$x_0$ as the initial values of all the variables in $U$ and $V$. 

We now let the vector of free variables in the columns 2 through $r-1$ of  $U$ take values in  $T$ and approach the projection of $x_0$ on $\mathbb R^c$,   let the free variables in $V$ take rational values close to their initial values in $x_0$. The entries of the last column of $U$ are also free variables allowed to take rational values close to their initial values in $x_0$. Obviously, when all the independent variables are assigned rational values, the dependent variables also have rational values. 

As rational functions are continuous whenever defined, we see that all the rational functions expressing the dependent variables are continuous at the rational point $x_0$, and the the entries of $UV$ are then continuous functions of the variables in $U$ (which are all free) and  and the independent variables in $V$. The fact that dependent variables in $V$ are solved from (3.3) ensures that $(UV)_{ij}=0 $ whenever $\bold a_{ij}=0$. Continuity of the rational functions guarantees that  sgn$(\tilde U \tilde V)_{ij}=$ $(\bold a_{ij})$ 
whenever all the free variables are sufficiently close to their initial values in $x_0$. Thus we can choose suitable rational values for the variables in the columns 2 through $r-1$ of  $U$ so that the resulting  vector in in $T$ and each variable is sufficiently close to its initial value in $x_0$, choose the free variables in the last column of $U$ or in $V$ to be some rational number sufficiently close to their initial values, and calculate the values of the dependent variables by using the rational functions expressing them in terms of the independent variable so that the resulting rational matrices   $\tilde U$ and $\tilde V$ satisfy that sgn$(\tilde U \tilde V)_{ij}=$ $(a_{ij})$ for all $i$ and $j$. Thus we arrive at a rational matrix $B =\tilde U \tilde V\in Q({\cal A})$ such that rank$(B) \leq r$. It follows that $mr_{\mathbb Q} ({\cal A}) = mr({\cal A})=r.$
\end{proof}

%

%
%
%
%

By applying the preceding theorem to the transpose of the condensed sign pattern of a sign pattern, we immediately get the following result.

\begin{corollary}\label{col:row-r-1}

Let $\cal A$ be any sign pattern and let $r= \text{mr}({\cal A})$. If the number of zero entries on each row of ${\cal A}_c$ is at most $r-1$, then $mr_{\mathbb Q} ({\cal A}) = mr({\cal A})$ .

\end{corollary}

Furthermore, by Lemma 2.4 in \cite{Arav05}, for every integer $k$ ($\text{mr}({\cal A}) \leq k \leq \mbox{MR}({\cal A})$) there exists a matrix $B \in Q({\cal A})$ with rank $k$.  By considering a full rank factorization of such a matrix $B$, the proof of theorem 3.2 can be adapted to establish the following result.

\begin{theorem}\label{thm:k0}

Let $\cal A$ be any sign pattern and let $r=mr({\cal A})$. If the number of zero entries on each column of $\cal A$ is at most $k$  for some integer $k$ with $r\leq k \leq \text{MR}({\cal A})-2$, then $mr_{\mathbb Q} ({\cal A}) \leq  k+1$ .

\end{theorem}

\section{The smallest known sign pattern whose minimum rank is 3 but whose rational minimum rank is greater than 3}

Kopparty and Rao \cite{KP08} showed the existence of a $12 \times 12$ sign pattern with  minimum  rank  3 and  rational minimum rank  greater than 3, suing the following configuration. 

\hspace{7.3cm}
\setlength{\unitlength}{0.8mm}
\begin{picture}(100,70)

\put(-10,10){\makebox(0,0){$\bullet$}}
\put(10,10){\makebox(0,0){$\bullet$}}
\put(-42,10){\makebox(0,0){$\bullet$}}
\put(42,10){\makebox(0,0){$\bullet$}}
\put(-16,29){\makebox(0,0){$\bullet$}}
\put(16,29){\makebox(0,0){$\bullet$}}
\put(0,23.7){\makebox(0,0){$\bullet$}}
\put(-24.5,58.5){\makebox(0,0){$\bullet$}}
\put(24.5,58.5){\makebox(0,0){$\bullet$}}

\put(-50,10){\line(1,0){100}}
\drawline(-42,10)(-16,29)(0,40.8)(28,61)
\drawline(-10,10)(-16,29)(-26,63)
\drawline(-42,10)(0,23.7)(16,29)
\drawline(10,10)(0,23.7)(-24.5,58.5)

\drawline(42,10)(16,29)(0,40.8)(-28,61)
\drawline(10,10)(16,29)(26,63)
\drawline(42,10)(0,23.7)(-16,29)
\drawline(-10,10)(0,23.7)(24.5,58.5)

\put(-42,7){\makebox(0,0){$p_1$}}
\put(-10,7){\makebox(0,0){$p_5$}}
\put(10,7){\makebox(0,0){$p_6$}}
\put(42,7){\makebox(0,0){$p_2$}}
\put(-20,29){\makebox(0,0){$p_7$}}
\put(20,29){\makebox(0,0){$p_8$}}
\put(0,27.7){\makebox(0,0){$p_9$}}
\put(-28.5,58.5){\makebox(0,0){$p_3$}}
\put(28.5,58.5){\makebox(0,0){$p_4$}}

\put(-30,7){\makebox(0,0){$l_1$}}
\put(-25,15){\makebox(0,0){$l_2$}}
\put(-30,23){\makebox(0,0){$l_3$}}
\put(25,15){\makebox(0,0){$l_4$}}
\put(30,23){\makebox(0,0){$l_5$}}
\put(-24,45){\makebox(0,0){$l_6$}}
\put(-12,45){\makebox(0,0){$l_7$}}
\put(24,45){\makebox(0,0){$l_8$}}
\put(12,45){\makebox(0,0){$l_9$}}

\put(0,-1){\makebox(0,0){Figure $4.1$}}

\end{picture}

\vspace{7mm}

\begin{example}
Using our approach in Section 2, we obtain the following $9 \times 9$ sign pattern ${\cal A}_0$ corresponding to the preceding point-line configuration, with $mr({\cal A}_0) \leq 3$.. 
$$
{\cal A}_0 = \begin{bmatrix} 
0 & 0 & 0 & - & - & - & - & + & +  \\
0 & - & - & 0 & 0 & + & + & - & -  \\
+ & + & + & + & 0 & 0 & 0 & + & +  \\
+ & + & 0 & + & + & + & + & 0 & 0  \\
0 & - & - & - & - & 0 & - & + & 0  \\
0 & - & - & - & - & + & 0 & 0 & -  \\
+ & + & 0 & 0 & - & 0 & - & + & +  \\
+ & 0 & - & + & 0 & + & + & 0 & -  \\
+ & 0 & - & 0 & - & + & 0 & + & 0
\end{bmatrix}.
$$
Note that the submatrix  ${\cal A}_0[\{4, 5, 6\}, \{7, 8, 9\}]$ is sign nonsingular. So $mr({\cal A}_0)=3$.  Furthermore, since Figure 4.1 cannot be achieved by using only rational points and rational lines (lines passing through 2 rational points) \cite{Grunbaum09}, $mr_{\mathbb Q}({\cal A}_0) > 3$. But it is easy to observe that after deleting the point $p_9$ from Figure 4.1, the resulting 8 point--9 line configuration  can be achieved by using only rational points and rational lines. Thus after deleting the last row of ${\cal A}_0$, the rational minimum rank of the new sign pattern, ${\cal A}_1$, is 3. Since deleting 1 row can decrease the rank of a matrix by at most 1, $mr_{\mathbb Q}({\cal A}_0) = 4$.  As indicated in \cite{Grunbaum09}, the 9 point--9 line configuration in figure 4.1 is probably the smallest point-line configuration that does not have rational realization.  Thus the sign pattern ${\cal A}_0$ is probably the smallest sign pattern whose minimum ranks over the reals and the rationals are different. 
\end{example}

\section{Open problems}

For the sign pattern ${\cal A}_0$ in the last example, there are  4 zeros in  column 1. It was shown in Theorem \ref{thm:2z} that for a sign pattern with minimum rank 3, if the number of zeros in each column is at most 2, then the sign pattern has rational minimum rank 3 as well. This leaves the case of having up to three zeros in each column open. The lack of known not rationally realizable point-line configurations on the plane in which each line contains at most three points in the configuration suggests the following conjecture.

\begin{conjecture}\label{coj:three0}
Let $\cal A$ be any sign pattern with $mr({\cal A}) = 3$. If the number of zero entries on each column of $\cal A$ is at most $3$, then
$\mbox{mr}_{\mathbb Q} ({\cal A}) =\mbox{mr}({\cal A})$.
\end{conjecture}

It should be mentioned that this conjecture is stronger than the conjecture in \cite{Grunbaum09} on the rational realizability of point-line configurations $C$ (referred to as 3-configurations in \cite{Grunbaum09}) on the plane such that each point in $C$ is on exactly three lines in $C$ and each line in $C$ passes through exactly three points in $C$.    Sturmfels and White  \cite{Sturmfels90} showed that every 11 point-11 line or 12 point-12 line configuration in which every point is contained in exactly 3 lines and every line passes through exactly 3 points, can be achieved by using only rational points and rational lines (lines passing through 2 rational points).  The general case is still open. 

Generalizing Conjecture \ref{coj:three0}, we have the following natural question. 

\begin{problem} \label{prob:r-zeros}
Let $\cal A$ be any sign pattern with $mr({\cal A}) = r\geq 3 $. If the number of zero entries on each column of $\cal A$ is at most $r$, it is w
always true that 
$\mbox{mr}_{\mathbb Q} ({\cal A}) =\mbox{mr}({\cal A})$?
\end{problem}

It is shown in \cite{Grunbaum09} that all points and lines in Figure 4.1 can be represented by using numbers in $\mathbb Q(\sqrt {5})$. We say that the configuration in Figure 4.1 requires $\sqrt {5}$. By mixing point--line configurations that require different irrational numbers, we suspect that one can construct a new configuration such that the corresponding sign pattern has real minimum rank 3 and large rational minimum rank.

A natural question is the following.

\begin{problem}
Is it true that for each integer $k\geq 4$, there exists a sign pattern $\cal A$ such that
$$\mbox{mr}({\cal A})=3 \quad \mbox{and}\quad \mbox{mr}_{\mathbb Q} ({\cal A}) =k? $$

\end{problem}

\end{document}